\documentclass[a4paper, 11pt]{amsart}
\usepackage{amsmath, amsthm, amssymb}
\usepackage{graphicx}
\usepackage[final,colorlinks]{hyperref}

\newcommand{\E}{\mathbb{E}}
\theoremstyle{plain}
\newtheorem{thm}{Theorem}
\newtheorem{lemma}{Lemma}
\theoremstyle{remark}

\numberwithin{equation}{section}

\title{Bounds for the expected value of one-step processes}

\author[B. Armbruster]{Benjamin Armbruster}

\author[\'A. Besenyei]{\'Ad\'am Besenyei}

\author[P. L. Simon]{P\'eter L. Simon}

\address{Department of Industrial Engineering and Management Sciences,
Northwestern University}
\email{armbruster@northwestern.edu}

\address{Institute of Mathematics, E\"otv\"os Lor\'and University Budapest\\
and Numerical Analysis and Large Networks Research Group, Hungarian Academy of
Sciences}
\email{badam@cs.elte.hu}
\email{simonp@cs.elte.hu}

\date{\today}



\thanks{\'A. B. and P.L.S were supported by OTKA grant no. 81403.}

\begin{document}

\begin{abstract}
Mean-field models are often used to approximate Markov processes with large state-spaces.  One-step processes, also known as birth-death processes, are an important class of such processes and are processes with state space $\{0,1,\ldots,N\}$ and where each transition is of size one.  We derive explicit bounds on the expected value of such a process, bracketing it between the mean-field model and another simple ODE.  While the mean-field model is a well known approximation, this lower bound is new, and unlike an asymptotic result, these bounds can be used for finite $N$.  Our bounds require that the Markov transition rates are density dependent polynomials that satisfy a sign condition.  We illustrate the tightness of our bounds on the SIS epidemic process and the voter model.
\end{abstract}

\maketitle

\section{Introduction}

Mean-field approximations of stochastic processes are crucial in many areas of science, where a large system is described by a stochastic process and its expected behaviour is approximated by a simpler mean-field model using a system of differential equations. The stochastic processes of particular interest to us are binary network processes, where each node of a large network can be in one of two states and the state of a node changes depending on the states of the neighbouring nodes.  The theory of their mathematical modelling can be found in several books and review papers \cite{Barratetal, Danon, Nekovee, Newmanetal}. The two well-known examples we analyze in Section 4 are epidemic and rumour spreading.


For concreteness we consider a continuous time Markov process, $X(t)$, with states $k\in\{0,1,\ldots,N\}$ and let $p_k(t)=P[X(t)=k]$ denote the discrete probability distribution of $X(t)$.  The time evolution of $p_k(t)$ is described by a linear system of differential equations, called the master equations. There is a well established theory for solving linear systems by expressing their solution as $\exp(At)$, where $A$ is the transition rate matrix of the system. It is also well-known that the matrix exponential is hard to compute for large matrices, hence methods have been developed exploiting any special structure of the matrix. Recently, an extremely powerful method has been worked out for tridiagonal matrices by Smith and Shahrezaei \cite{SmithShahrezaei}. Since one-step processes have tridiagonal transition matrices, the time dependence of the probabilities $p_k(t)$ can be computed very efficiently. Besides computational methods there are theoretical approaches to approximate, estimate and characterize qualitatively the time dependence of the probability distribution $p_k$ and its moments (mainly its expected value). The typical approach is to introduce a low dimensional system of non-linear ODEs called the mean-field equations. These are not only faster to solve computationally but also allow for analysis giving a better qualitative understanding of the system.

The accuracy of mean-field approximations is not obvious. The problem of rigorously linking exact stochastic models to mean-field approximations goes back to the early work of Kurtz \cite{Kurtz1970} (see \cite{EthierKurtz} for a more recent reference). He studied density-dependent Markov processes and proved their stochastic convergence to the deterministic mean-field model. It can be shown that the difference between the solution of the mean-field equation and the expected value of the process is of order $1/N$ as $N$ tends to infinity \cite{BaKiSiSi}.

However, it is natural to look for actual lower and upper bounds that can be used for finite $N$ (in contrast to the previous asymptotic results). It is known that in many cases, the mean-field model yields an upper bound on the expected value of the process. The first lower bound is by Armbruster and Beck \cite{ArBe} where a system of two ODEs proves a lower bound on the expected value in the case of a susceptible-infected-susceptible (SIS) epidemic model on a complete graph. The aim of this paper is to extend these results to a wider class of Markov chains, including several which approximate network processes.

We develop bounds for \emph{one-step processes}, also known as \emph{birth-death processes}, that form an important class of continuous time Markov processes, for which the state space is $\{ 0,1, \ldots , N\}$ and it is assumed that transition from state $k$ is possible only to states $k-1$ at rate $c_k$ and to state $k+1$ at rate $a_k$. The master equations are
\begin{equation}
    p_k'= a_{k-1}p_{k-1}-(a_{k}+c_{k})p_k+c_{k+1}p_{k+1},\quad k=0,\ldots,  N , \label{eqKolm}
\end{equation}
where of course $a_{-1}=a_N=c_0=c_{N+1}=0$.
The bounds we develop are the solutions of simple ODEs and will hold for all $N$.  We will illustrate with numerical examples that the upper and lower bounds are remarkably close to each other.

In Section 2 we set up the problem, give the mean-field equation and the approximating system, state the main result, and state the tools used in the proof.  Section 3 proves the main result.  In Section 4 we illustrate the performance of the bounds on two examples: an SIS epidemic with and without airborne infection, and a voter-like model. Section 5 concludes and discusses directions for future research.


\section{Formulation of the main result}

\subsection{Differential equations for the moments}
We first establish the differential equations for the moments of the process $X(t)$. We focus on the fraction $X/N$ since we are interested in situations with large $N$ and density dependent coefficients. By definition, the $n$-th moment of $(X/N)$ is
\begin{equation}\label{ynODE}
y_n(t) := \E[(X(t)/N)^n]=\sum_{k=0}^N (k/N)^n p_k(t)\quad (n=0,1,\dots).
\end{equation}
Of course $y_0=1$.
The following equations for $y_n'$ can be derived using Lemma 2 in \cite{BaKiSiSi}, using the Kolmogorov backward equations, or taking the time derivative of \eqref{ynODE} and substituting \eqref{eqKolm}.
\[
y'_n = \frac{1}{N^n} \sum_{k=0} ^N \left( a_k ( (k+1)^n-k^n) + c_k ( (k-1)^n-k^n) \right) p_k .
\]
We are interested in the case when $a_k/N$ and $c_k/N$ are polynomials of $k/N$. Then, $y'_n$ can be expressed in terms of the coefficients of these polynomials. In fact, the following result was proved in \cite{BaKiSiSi}.

\begin{lemma}\label{momlem}
Let
\begin{equation}\label{akck}
\frac{a_k}N=A(k/N) \quad\text{and}\quad\frac{c_k}N=C(k/N)
\end{equation}
with polynomials $A(x)=\sum_{j=0}^mA_jx^j$ and $C(x)=\sum_{j=0}^mC_jx^j$ such that $A(1)=0$ and $C(0)=0$.
Then
\begin{align}\label{ysystem1}
y_1'&=\sum_{j=0}^mD_jy_j,\\\label{ysystem2}
y_n'&=n\sum_{j=0}^mD_jy_{n+j-1}+\frac1NR_n\quad (n=2,3,\dots)
\end{align}
where $D_j=A_j-C_j$, $0\leq R_n\leq \frac{n(n-1)}2 c$, and $c=\sum_{j=0}^m (|A_j|+|C_j|)$.
\end{lemma}

\subsection{Mean-field equation and an approximating system}
Our mean-field equation
\begin{equation}\label{meanf}
y'=\sum_{j=0}^mD_jy^j,\quad y(0)=y_1(0),
\end{equation}
is motivated by the approximations $\frac1NR_n\approx0$ and $y_n\approx y_1^n$ for large $N$,
where the second approximation essentially assumes $X/N$ is deterministic. We are interested whether the solution $y$ of the mean-field equation converges uniformly on $[0,T]$ to $y_1$ as $N\to\infty$. Recently in \cite{ArBe}, the following special case was proved in an elementary way.
\begin{thm}\label{arbethm}
Let $m=2$, $D_0=0$, $D_1=\tau-\gamma$, and $D_2=-\tau$ with positive constants $\tau$ and $\gamma$.
If $y(0)=y_1(0)=u\in [0,1]$ is fixed, then for any fixed $T>0$,
\[y_1(t)\leq y(t)\quad \text{ for } t\in[0,T],\]
and
\[\lim_{N\to\infty}|y(t)-y_1(t)|=0\quad \text{uniformly in } t\in [0,T].\]
\end{thm}
In the theorem above, the constants $\tau$ and $\gamma$ correspond to the infection and recovery rates in an SIS model of disease spread (see Subsection \ref{sissub}). Our aim is to generalize Theorem \ref{arbethm} to a broader class of coefficients $D_j$ and to provide a lower bound for $y_1$. In the proof of our main result we will not only compare $y$ with $y_1$ but also $y^n$ with $y_n$.  Thus using \eqref{meanf},
\[(y^n)'=ny^{n-1}y'=n\sum_{j=0}^mD_jy^{n+j-1}=n\sum_{j=0}^mD_j(y^n)^{\frac{n+j-1}n}.\]
Hence, the powers of $y$ satisfy the initial value problem below:
\begin{align}\label{psystem1}
y'&=\sum_{j=0}^mD_jy^j,\quad y(0)=y_1(0),\\\label{psystem2}
(y^n)'&=n\sum_{j=0}^mD_j(y^n)^{\frac{n+j-1}n},\quad y^n(0)=y_1^n(0)\quad (n=2,3,\dots).
\end{align}
(The equation for $y'$ is separated from $(y^n)'$ for $n\geq 2$ because it will have a different role.)
This system in combination with system \eqref{ysystem1}--\eqref{ysystem2} for $y_n'$ motivates the following initial value problem:
\begin{align}\label{zsystem1}
z_1'&=\sum_{j=0}^mD_jz_j,\quad z_1(0)=y_1(0),\\\label{zsystem2}
z_n'&=n\sum_{j=0}^mD_jz_n^{\frac{n+j-1}n}+\frac{n(n-1)}{2N}c,\quad z_n(0)=y_1^n(0)\quad (n=2,\dots,m),
\end{align}
where we let $z_0=1$.

\subsection{Main result}
We are now ready to state our main result.
\begin{thm}\label{main}
Assume that
\begin{equation}\label{sign}
D_0\geq0, D_1\in\mathbb{R}\text{ and } D_j\leq0 \text{ for }j\geq 2,
\end{equation}
and let $y_1(0)=u\in(0,1]$ be fixed.
Then for the solutions $y$ of \eqref{meanf} and $z_1$ of \eqref{zsystem1}--\eqref{zsystem2}, it holds that
\begin{align*}
z_1(t)\leq y_1(t)\leq y(t)& \text{ for }t\geq0,\\
z_1(t)^n \leq y_n(t)\leq z_n(t)& \text{ for }t\geq0,\ 2\leq n\leq m,
\end{align*}
and for every $T>0$ there exists a constant $C_T>0$ such that
\[
z_1(t)-y(t)\leq \frac{C_T}N \text{ in }[0,T].
\]
\end{thm}

The proof is based on some familiar inequalities which we recall in the next subsection.

\subsection{Tools of the proof}

The following comparison results are standard in the theory of ODEs, see \cite{Hale}.

\begin{lemma}[Comparison]\label{complem}
Suppose that $f(t,x)$ is continuous in $x$;
\begin{itemize}
\item the initial value problem $x_2'(t)=f(t,x_2(t))$, $x_2(0)=x_0$ has a unique solution for $t\in[0,T]$;
\item $x_1'(t)\leq f(t,x_1(t))$ for $t\in[0,T]$; and $x_1(0)\leq x_0$.
\end{itemize}
Then $x_1(t)\leq x_2(t)$ for $t\in[0,T]$.
\end{lemma}

\begin{lemma}[Peano's inequality]
Suppose that $f_1,f_2\colon[0,T]\times[a,b]\to\mathbb{R}$ are Lipschitz continuous functions in their second variable with Lipschitz constant $L$ and $|f_1(t,x)-f_2(t,x)|\leq M$ in $[0,T]\times[a,b]$ with some constant $M$. If
\begin{itemize}
\item $x_1'(t)=f_1(t,x_1(t)), x_2'(t)=f_2(t,x_2(t))$ for $t\in(0,T]$ and
\item $x_1(0)=x_2(0)$,
\end{itemize}
then
\[|x_1(t)-x_2(t)|\leq \frac{M}L\left(e^{Lt}-1\right) \quad (t\in[0,T]).\]
\end{lemma}

The classical Jensen's inequality and the definition of the expected value yields the probabilistic version of Jensen's inequality, see \cite{Ross}.

\begin{lemma}[Jensen's inequality]
If $X$ is a random variable and $\varphi\colon\mathbb{R}\to\mathbb{R}$ is a convex function, then
\[\varphi(\mathbb{E}[X])\leq\mathbb{E}[\varphi(X)].\]
For concave $\varphi$, the reverse inequality holds.
\end{lemma}


\section{Proof of the main result}

\begin{proof}[Proof of Theorem \ref{main}]
The proof will be carried out in multiple steps. First, we derive bounds for $y_n$ and $z_n$ independent of $N$ which will guarantee uniform Lipschitz constants with respect to $N$ in steps three and four. Second, we show that $y_1\leq y$. In the third step, we show that $y_n \leq z_n$ for $n\geq 2$ and use Peano's inequality to deduce from equations \eqref{psystem2}, \eqref{ysystem2} and \eqref{zsystem2} that there is some constant $C(n,T)$ depending on $n$ and $T$ such that
\[|y^n(t)-z_n(t)|\leq \frac{C(n,T)}N \text{ for } t\in[0,T],\ n=2,\ldots,m.\]
Finally with these estimates at hand, we obtain in the same manner $z_1\leq y_1$ and
\[|y(t)-z_1(t)|\leq \frac{C_T}N \text{ for } t\in[0,T],\]
with some suitable constant $C_T$.

\textbf{Step 1: A priori bounds independent of $N$.}
We first focus on a lower bound for $y_n$.
The definition of $y_n$ implies $0\leq y_\ell\leq y_n$ for all $n\leq \ell$. Therefore,
\[y_1'\geq D_0+y_1\sum_{j=1}^mD_j\geq y_1 D,\ y_1(0)=u,\]
where the second inequality is due to \eqref{sign} and we define $D=\sum_{j=1}^mD_j$.  Applying Lemma \ref{complem} we obtain
\[y_1(t)\geq ue^{tD}.\]
Now due to Jensen's inequality,
\begin{equation}\label{Jensen}
y_1^n=\E[X/N]^n\leq \E[(X/N)^n]=y_n\text{ for } n\geq 1.
\end{equation}
Defining $\delta_1(T)=u^m\min\{1,e^{TmD}\}>0$, this leads to our lower bound,
\begin{equation}\label{also}
y_n(t) \geq y_1^n(t) \geq u^n e^{tnD}\geq \delta_1(T)\text{ for } 1\leq n\leq m,t\in[0,T].
\end{equation}

Next we focus on an upper bound for $z_n$ for $n\geq 2$.  We bound the definition \eqref{zsystem2} by assuming $z_n\geq 1$. Then using the sign condition \eqref{sign}, $z_n^{\frac{n-1}{n}}\leq z_n$, and $N\geq 1$,
\[z'_n \leq n(D_0+|D_1|+c n)z_n,\quad z_n(0)=u^n \text{ for } n\geq 2.\]
Applying Lemma \ref{complem} to this differential inequality, we obtain an upper bound,
\begin{equation}\label{zbound}
z_n(t) \leq u^n e^{n(D_0+|D_1|+cn)t}\leq \delta_2(T)\text{ for } 2\leq n\leq m,t\in[0,T],
\end{equation}
where $\delta_2(T)=\max\{1,u^2 e^{m(D_0+|D_1|+cm)T}\}$.

\textbf{Step 2: Comparison of $y$ and $y_1$.}
Substituting \eqref{Jensen} into \eqref{ysystem1} with regard to the sign condition \eqref{sign} yields
\[y_1'\leq \sum_{j=0}^m D_jy_1^j.\]
Applying Lemma \ref{complem} to the above inequality and to \eqref{meanf} yields for $t\in[0,T]$,
\begin{equation}\label{y1lessy}
y_1(t)\leq y(t).
\end{equation}

\textbf{Step 3: Comparison of $y_n$, $y^n$, and $z_n$.}
Applying Jensen's inequality again,
\begin{gather}\label{jen1}
y_{n+j-1}=\E[(X/N)^{n+j-1}]\geq \E[(X/N)^n]^{\frac{n+j-1}n}=y_n^{\frac{n+j-1}n}\text{ for } j\geq 1,\\
y_{n-1}=\E[(X/N)^{n-1}]\leq \E[(X/N)^n]^{\frac{n-1}n}=y_n^{\frac{n-1}{n}}.\label{jen2}
\end{gather}
We also have the trivial estimate
\begin{equation}\label{triv}
0\leq y_n\leq1.
\end{equation}

Putting the estimates \eqref{jen1}, \eqref{jen2} and \eqref{triv} into the differential equation \eqref{ysystem2} of the $n$-th moment for $n\geq 2$ and taking care of the sign condition \eqref{sign} we obtain
\begin{equation}\label{ynest}
y_n'\leq n\sum_{j=0}^mD_jy_n^{\frac{n+j-1}n}+\frac{n(n-1)}{2N}c,\quad y_n(0)=y_1^n(0)=u^n.
\end{equation}
By introducing the function $g_n(x)=n\sum_{j=0}^mD_jx^{\frac{n+j-1}n}+\frac{n(n-1)}{2N}c$, \eqref{ynest} can be written as $y_n'\leq g_n(y_n)$, and \eqref{zsystem2} has the form $z_n'=g_n(z_n)$.
Now, $g_n(x)$ is Lipschitz continuous except at 0. Thus we can apply Lemma \ref{complem} while $z_n\geq\delta_1(T)$ holds, to obtain
\begin{equation}\label{ysubnzn}
y_n(t)\leq z_n(t) \text{ for } 2\leq n\leq m,t\in [0,T].
\end{equation}
Fortunately, \eqref{also} ensures $z_n(t)\geq y_n(t)\geq \delta_1(T)$ for $t\leq T$ and $2\leq n\leq m$.

Using the same steps and \eqref{psystem2} we can also show that
\begin{equation}\label{ynzn}
y^n(t)\leq z_n(t) \text{ for } 2\leq n\leq m,t\in [0,T].
\end{equation}
Now,
\begin{equation}\label{zybound}
\delta_2(T)\geq z_n(t)\geq y^n(t) \geq y_1^n(t) \geq \delta_1(T)
\text{ for } 2\leq n\leq m,t\in [0,T],
\end{equation}
where the inequalities are from \eqref{zbound}, \eqref{ynzn}, \eqref{y1lessy}, and \eqref{also}, respectively.

This lets us use Peano's inequality to estimate the difference of the solution $y^n$ of \eqref{psystem2} and $z_n$ of \eqref{zsystem2} since $z_n'=g_n(z_n)$, $(y^n)'=g_n(y^n)-\frac{n(n-1)}{2N}c$, and \eqref{zybound} ensures a Lipschitz constant independent of $N$. Therefore, there is some constant $C(n,T)$ depending on $n$ and $T$ such that
\begin{equation}\label{est}
|y^n(t)-z_n(t)|\leq \frac{C(n,T)}N \text{ for } 2\leq n\leq m,t\in[0,T].
\end{equation}

\textbf{Step 4: Comparison of $y_1$ and $z_1$.}
Now, with estimates \eqref{ysubnzn} and \eqref{est} in our hand we turn to equations \eqref{psystem1}, \eqref{ysystem1}, \eqref{zsystem1} to obtain estimates on their solutions analogously to the preceding part of the proof. First, by substitution of estimates \eqref{ysubnzn} into \eqref{ysystem1} with regard to the sign condition \eqref{sign} it follows that
\begin{equation}\label{y1est}
y_1'\geq D_0+D_1y_1+\sum_{j=2}^mD_jz_j,\quad y(0)=y_1(0)=u.
\end{equation}
Considering functions $z_n$ ($n=2,\dots,m$) fixed, then \eqref{y1est} and \eqref{zsystem1} have the form $y_1'(t)\geq g_1(t,y_1(t))$ and $z_1'(t)=g_1(t,z_1(t))$ where $g_1(t,x)=D_0+D_1x+\sum_{j=2}^mD_jz_j(t)$. Thus, Lemma \ref{complem} implies
\[y_1\geq z_1\]
and this leads to the lower bound for the $n$-th moment as $y_n\geq z_1^n$ by using \eqref{also}.

Now we apply Peano's inequality to estimate the difference of the solution $y$ of \eqref{psystem1} and $z_1$ of \eqref{zsystem1}. Indeed, $y'(t)=h_1(t,y(t))$ and $z_1'(t)=g_1(t,z_1(t))$ where $h_1(t,x)=D_0+D_1x+\sum_{j=2}^mD_jy^j(t)$. Therefore,
\[h_1(t,x)-g_1(t,x)=\sum_{j=2}^m D_j(y^j(t)-z_j(t)).\]
Thus by \eqref{est},
\[| h_1(t,x)-g_1(t,x)|\leq \frac{m \cdot \max_{2\leq n\leq m} |D_n|C(n,T)}N \text{ for } t\in[0,T].\]
Then Peano's inequality implies that there is some constant $C$ depending on $m$ and $T$ such that
\[
|y(t)-z_1(t)|\leq \frac{C}N \text{ for } t\in[0,T].
\]
The proof is now complete.
\end{proof}


\section{Examples}

Here we show numerically the lower and upper bounds in the case of two network processes.

\subsection{SIS epidemic propagation with and without airborne infection}\label{sissub}

Consider SIS epidemic propagation on a regular random graph with $N$ nodes.  We allow for multiple routes of infection so that the infection spreads not only via the contact network, but also due to external forcing, such as airborne infection. The state space of the corresponding one-step process is $\{ 0,1, \ldots , N\}$, where $k$ denotes the state with $k$ infected nodes. In fact, the position of the infected nodes also affects the spreading process, hence the state space is larger, and therefore our model is only an approximation of the real infection propagation process. The validity of this approximation is discussed in detail in \cite{akckcikk}. We note that in the case of a complete graph, the model is exact. Starting from state
$k$ the system can move either to state $k+1$ or to $k-1$, since at a given instant only one node can change its state. When the system moves from state $k$ to $k+1$ then a susceptible node becomes infected. The rate of external infection of a susceptible node is denoted by $\beta$. The rate of internal infection is proportional to the number of infected neighbours, which is $d\frac{k}{N}$ in average, where $d$ denotes the degree of each node in the network and $\frac{k}{N}$ is the proportion of the infected nodes. Since there are $N-k$ susceptible nodes and each of them has $d\frac{k}{N}$ infected neighbours the total number of $SI$ edges is $d(N-k)\frac{k}{N}$. Then the rate of transition from state $k$ to state $k+1$ is obtained by adding the rates of the two infection processes
$$
a_k=\tau d(N-k)\frac{k}{N} +\beta (N-k) .
$$
where $\tau$ is the infection rate. The rate of transition from state $k$ to $k-1$ is $c_k=\gamma k$, because any of the $k$ infected nodes can recover with recovery rate $\gamma$. Using these coefficients, $a_k$ and $c_k$, the spreading process can be described by equation \eqref{eqKolm}. The coefficients can be given in the form \eqref{akck} by choosing the functions $A(x)=\tau d x (1-x)+\beta (1-x) $ and $C(x)=\gamma x$. The coefficients of these polynomials are $A_0=\beta $, $A_1=\tau d -\beta$, $A_2=-\tau d $, $C_0=0$, $C_1=\gamma$ and $C_2=0$. Thus the coefficients
$$
D_0=\beta, \quad D_1=\tau d -\beta - \gamma, \quad D_2= -\tau d
$$
satisfy the sign condition \eqref{sign}. According to \eqref{meanf}, the mean-field equation takes the form
\begin{equation}
y'=\beta+(\tau d -\beta - \gamma)y-\tau d y^2 \label{meanfSISa}
\end{equation}
subject to the initial condition $y(0)=i/N$, where $i$ is the number of initially infected nodes.
System \eqref{zsystem1}-\eqref{zsystem2} can be written as
\begin{align}\label{zsystem1SISa}
z_1'&=\beta+(\tau d -\beta - \gamma)z_1-\tau d z_2,\\\label{zsystem2SISa}
z_2'&=2\beta z_2^{1/2}+2(\tau d -\beta - \gamma)z_2 -2\tau d z_2^{3/2} + \frac{c}{N},
\end{align}
where $c=\beta+|\tau d-\beta|+\tau d +\gamma$. The initial condition is $z_1=i/N$, $z_2=(i/N)^2$. The mean-field equation \eqref{meanfSISa} and the system \eqref{zsystem1SISa}--\eqref{zsystem2SISa} can be easily solved with an ODE solver.

The solutions without airborne infection, $\beta=0$, are shown in Figure \ref{fig:SIS} for $N=10^6$ and $N=10^7$. It is important to note that for such large values of $N$ the master equation \eqref{eqKolm} cannot be solved numerically, but we know that the expected value $y_1(t)= \sum_{k=0}^N  \frac{k}{N}  p_k(t)$ is between the two curves given in the Figure. We can also see that for small times the two bounds are nearly identical, i.e., we get the expected value with high accuracy. As time increases the bounds move apart, moreover the length of time interval, where the two bounds give the expected value accurately increases with $N$.

The solutions with airborne infection, $\beta>0$, are shown in Figure \ref{fig:SISair} for $N=100$ together with the expected value $y_1(t)= \sum_{k=0}^N  \frac{k}{N} p_k(t)$ obtained by solving the master equation \eqref{eqKolm} for $p_k$. One can see that the expected value is between the two bounds, in fact, it is hardly to distinguish from the solution of the mean-field equation, therefore the stationary part of the curves are enlarged in the inset. Note that the performance of the bounds is much better than in the case without airborne infection. Here we get much closer bounds even for a small value of $N$.

\subsection{A voter-like model}

Consider again a regular random network where each node can be in one of two states, 0 or 1, representing two opinions propagating along the edges of the network (see \cite{HolleyLiggett}). If a node is in state 0 and has $j$ neighbours in state 1, then its state will change to 1 with probability $j\tau \Delta t$ in a small time interval $\Delta t$. This describes a node switching to opinion 1. The opposite case can also happen, that is a node in state 1 can become a node with opinion 0 with a probability $j\gamma \Delta t$ in a small time interval $\Delta t$, if it has $j$ neighbours in state 0. The parameters $\tau$ and $\gamma$ characterize the strengths of the two opinions. Voter models are related to the famous Ising spin model in physics where the atomic spin, $\pm1$, in a domain is affected by the spin in neighboring domains. The state space of the corresponding one-step process is $\{ 0,1, \ldots , N\}$, where $k$ denotes the state, in which there are $k$ nodes with opinion 1. Starting from state $k$ the system can move either to state $k+1$ or to $k-1$, since at a given instant only one node can change its opinion. When the system moves from state $k$ to $k+1$ then a node with opinion 0 is ``invaded'' and becomes a node with opinion 1. The rate of this transition is proportional to the number of neighbours with opinion 1, which is $d\frac{k}{N}$ in average, where $d$ is the degree of each node in the network and $\frac{k}{N}$ is the proportion of the nodes with opinion 1. Since there are $N-k$ nodes with opinion 0 and each of them has $d\frac{k}{N}$ neighbours with the opposite opinion the total number of edges connecting nodes with two different opinions is $d(N-k)\frac{k}{N}$. Hence the rate of transition from state $k$ to state $k+1$ is
$$
a_k=\tau d(N-k)\frac{k}{N}.
$$
Similar reasoning leads to the rate of transition from state $k$ to $k-1$ as
$$
c_k=\gamma dk\frac{N-k}{N}.
$$
Using these coefficients, $a_k$ and $c_k$, the spreading process can be described by equation \eqref{eqKolm}. The coefficients can be given in the form \eqref{akck} by choosing the functions $A(x)=\tau d x (1-x)$ and $C(x)=\gamma d x(1-x)$. The coefficients of these polynomials are $A_0=0$, $A_1=\tau d $, $A_2=-\tau d $, $C_0=0$, $C_1=\gamma d$ and $C_2=-\gamma d$. Thus the coefficients
$$
D_0=0, \quad D_1=\tau d - \gamma d, \quad D_2= \gamma d-\tau d
$$
satisfy the sign condition \eqref{sign} if $\gamma < \tau$. According to \eqref{meanf} the mean-field equation takes the form
\begin{equation}
y'=(\tau d - \gamma d)(y- y^2) \label{meanfvot}
\end{equation}
subject to the initial condition $y(0)=i/N$, where $i$ is the number of nodes with opinion 1 at time 0.
System \eqref{zsystem1}-\eqref{zsystem2} can be written as
\begin{align}\label{zsystem1vot}
z_1'&=(\tau d - \gamma d)z_1-(\tau d - \gamma d) z_2,\\\label{zsystem2vot}
z_2'&=2(\tau d - \gamma d)z_2 -2(\tau d - \gamma d) z_2^{3/2} + \frac{c}{N},
\end{align}
where $c=2\tau d +2\gamma d $. The initial condition is $z_1= i/N$, $z_2= (i/N)^2$. The mean-field equation \eqref{meanfvot} and system \eqref{zsystem1vot}--\eqref{zsystem2vot} can be easily solved with an ODE solver. The solutions are shown in Figure \ref{fig:vot} both for $D_2<0$ (left panel) and for $D_2>0$ (right panel). For $D_2<0$ the lower bound performs well only for large values of $N$. The master equation \eqref{eqKolm} cannot be solved numerically  for such large values of $N$. Hence the expected value $y_1$ is not shown in the left panel. For $D_2>0$, i.e., when $\gamma > \tau$, the role of $y$ and $z_1$ is exchanged. The solution $y$ of the mean-field equation becomes the lower bound and $z_1$ becomes the upper bound. The function $y$ is hardly distinguishable from the expected value $y_1$, the inset shows that $y$ is really a lower bound.

\begin{figure}[h!]
\begin{center}
\includegraphics[scale=0.4]{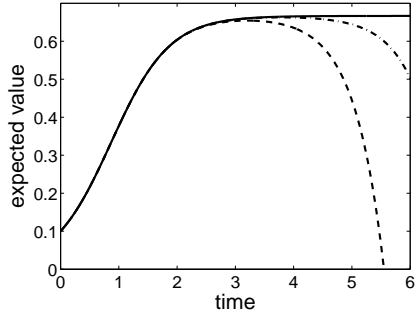}
\caption{SIS epidemic without airborne infection. The solution $y$ of the mean-field equation \eqref{meanfSISa} (continuous curve) and the first coordinate $z_1$ of the solution of system \eqref{zsystem1SISa}--\eqref{zsystem2SISa} for $N=10^6$ (dashed curve) and for $N=10^7$ (dashed-dotted curve). The parameter values are $\gamma=1$, $\tau=0.1$, $d=30$, and $\beta=0$. }\label{fig:SIS}
\end{center}
\end{figure}

\begin{figure}[h!]
\begin{center}
\includegraphics[scale=0.4]{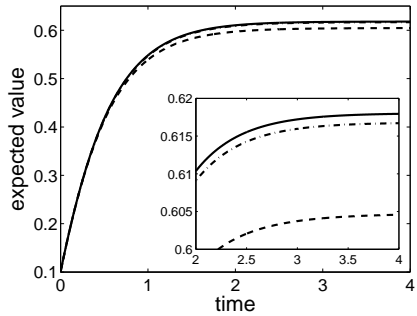}
\caption{SIS epidemic with airborne infection. The solution $y$ of the mean-field equation \eqref{meanfSISa} (continuous curve), the first coordinate $z_1$ of the solution of system \eqref{zsystem1SISa}--\eqref{zsystem2SISa} (dashed curve) and the expected value $y_1$ obtained from the solution of the master equation (dashed-dotted curve). The stationary part of the curves are enlarged in the inset. The parameter values are $N=100$, $\gamma=1$, $\tau=0.05$, $d=20$ and $\beta=1$. }\label{fig:SISair}
\end{center}
\end{figure}

\begin{figure}[h!]
\begin{center}
\includegraphics[scale=0.4]{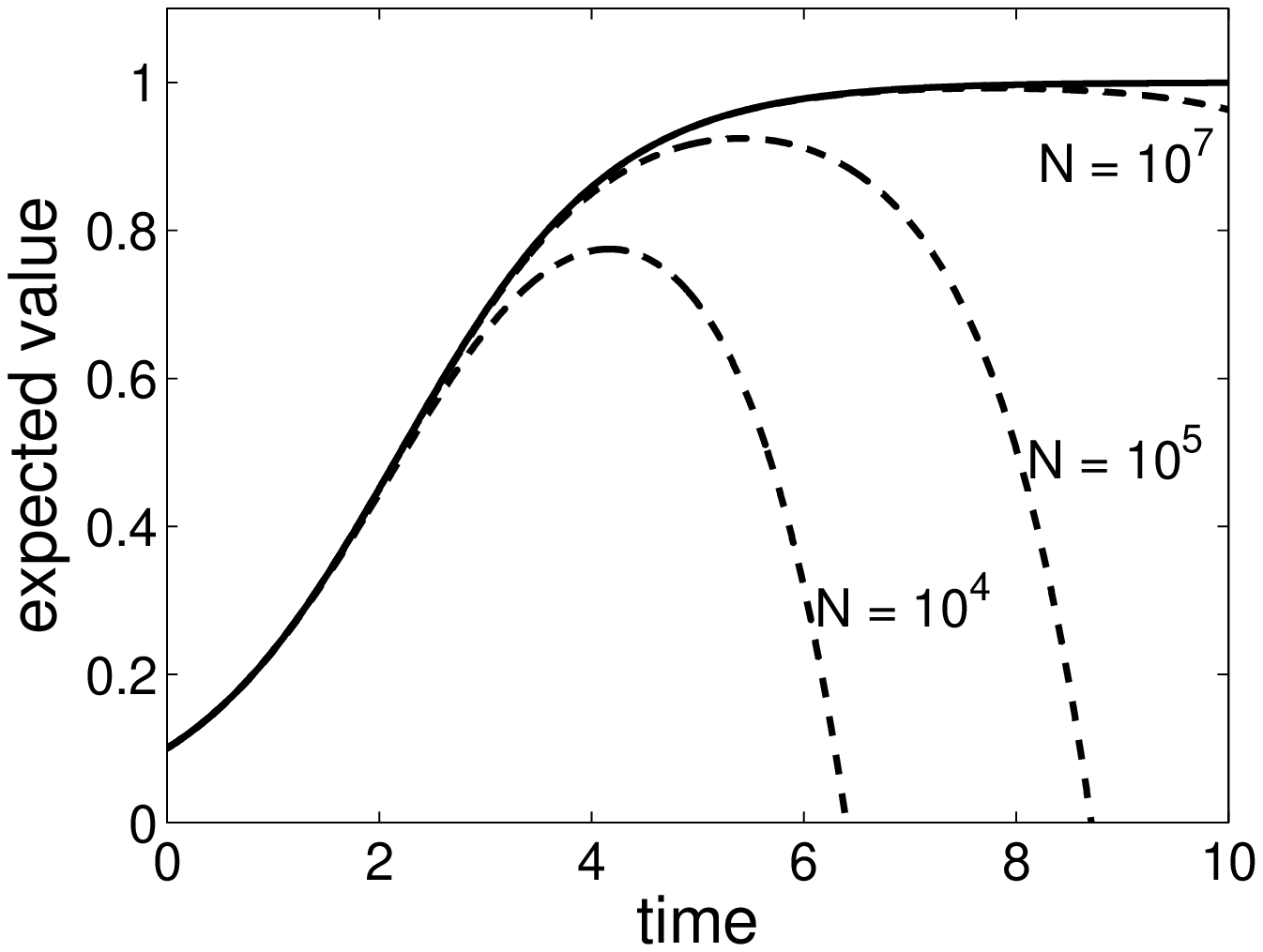}
\includegraphics[scale=0.4]{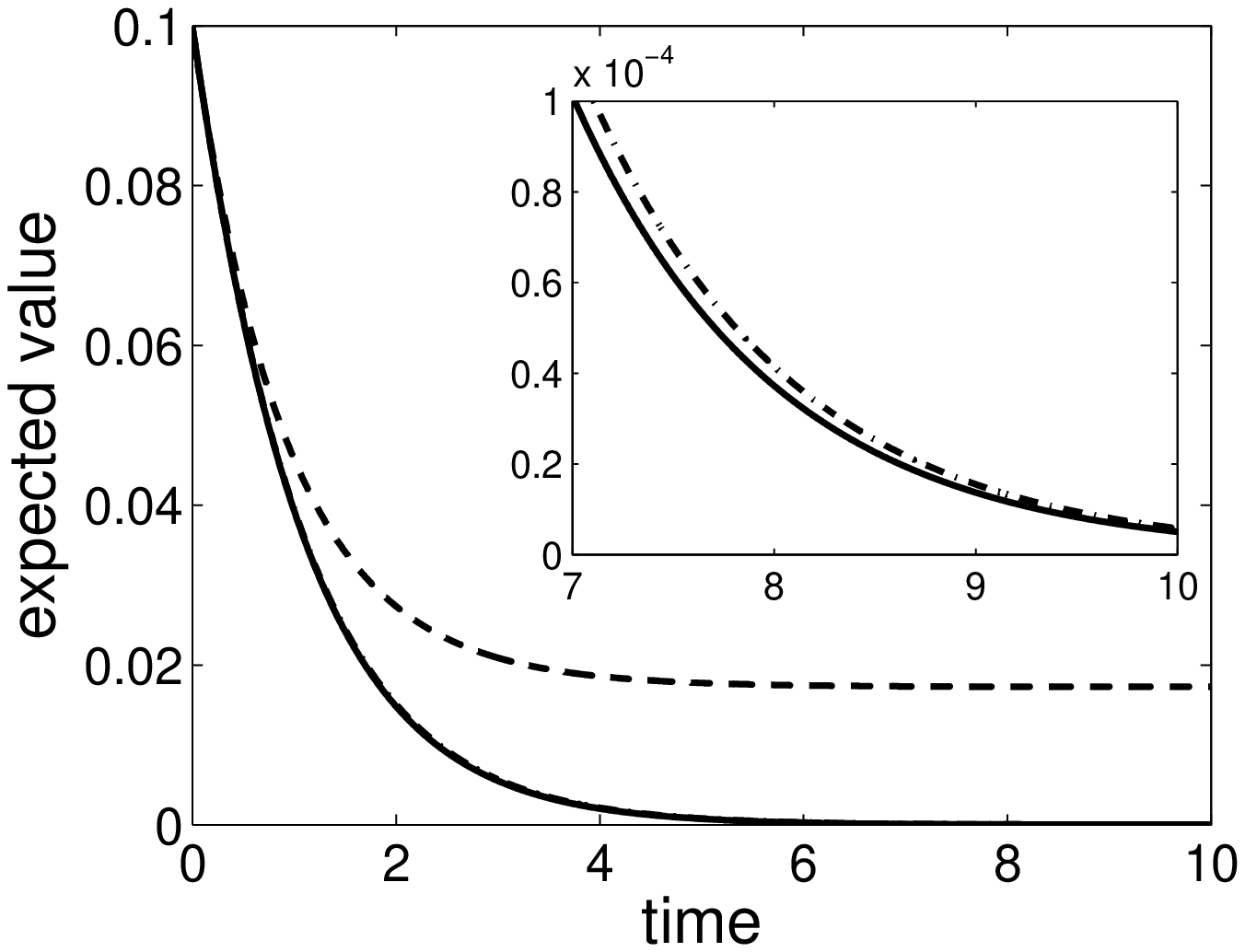}
\caption{Voter-like model. The solution $y$ of the mean-field equation \eqref{meanfvot} (continuous curve), the first coordinate $z_1$ of the solution of system \eqref{zsystem1vot}--\eqref{zsystem2vot} (dashed curve) and the expected value $y_1$ obtained from the solution of the master equation (dashed-dotted curve). Left panel: the case $D_2<0$, the values of $N$ are shown in the figure, the other parameter values are $\gamma =0.1$, $\tau=0.2$, $d=10$. Right panel: the case $D_2>0$, the parameter values are $N=200$, $\gamma =0.2$, $\tau=0.1$, $d=10$. The stationary part of the curves are enlarged in the inset. }\label{fig:vot}
\end{center}
\end{figure}

\begin{figure}[h!]
\begin{center}
\includegraphics[scale=0.4]{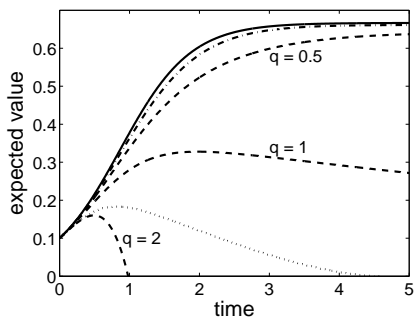}
\caption{Potential lower bounds of SIS model without airborne infection. The solution $y$ of the mean-field equation \eqref{meanfSISa} (continuous curve), the first coordinate $z_1$ of the solution of system \eqref{zsystem1SISq}--\eqref{zsystem2SISq} (dashed curve) for different values of $q$ shown in the figure, $z_1$ of the solution of system \eqref{zsystem1SISCS}--\eqref{zsystem2SISCS} (dotted curve) and the expected value $y_1$ obtained from the solution of the master equation (dashed-dotted curve). The parameter values are $N=100$, $\gamma =1$, $\tau=0.1$, $d=30$, and $\beta=0$. }\label{fig:SISq}
\end{center}
\end{figure}

\section{Discussion}

We started from the master equation of a one-step process, assumed that the coefficients are density dependent, and the functions $A$ and $C$ are polynomials. Then under certain sign condition on the coefficients of the polynomials we proved that the mean-field equation yields an upper bound for the expected value of the process. We constructed an auxiliary system for the artificially defined functions $z_j$, and proved that $z_1$ is a lower bound. We showed several examples, where the upper and lower bounds are close to each other, hence the method can be used to approximate the expected value without solving the large system of master equations or using simulation, which only gives probabilistic guarantees.

Two avenues for future research are relaxing the sign condition \eqref{sign} and improving the lower bound.
It is easy to see that in the case $D_j \geq 0$ for $j\geq 2$ and following the argument in Step 2 that the mean field equation yields a lower bound. Also in the voter-like model, a violation of the sign condition (i.e., when $\gamma$, the rate of switching to opinion 0, is greater than $\tau$, the rate of switching to opinion 1) can be dealt with by switching the labels of the opinions or equivalently replacing $x$ by $1-x$.  Perhaps more general violations of the sign condition can be dealt with by considering the convexity or concavity of the entire polynomials $A(x)$ and $C(x)$ instead of by the signs of their coefficients.

The second avenue for future research is improving the lower bound. Unlike the case of SIS disease propagation with airborne infection, $\beta>0$, where the upper and lower bounds are very close to each other even for small system sizes, say $N=100$, we can see that for the case of regular SIS epidemic without airborne infection, $\beta=0$, and the voter-like model, that the lower bound may be quite far from the expected value for moderately large $N$, say $N=10^5$.

The reason that the lower bound veers off to 0 is that $z_2$ converges to a positive steady state and then the derivative of $z_1$ becomes negative after some time. This problem can be overcome by altering the differential equation of $z_1$. The term $D_2z_2$ could be changed to a term that contains also $z_1$ in order to prevent $z_1$ from becoming negative. The new term could be introduced by exploiting the fact that $z_j$ approximates the $j$-th moment $y_j$ and this function can be approximated by $y^j$. This suggests that $z_2$ can be approximated by $z_1^2$, or in other words, $z_1$ approximates $\sqrt{z_2}$. In order to tune the approximation we introduce the artificial parameter $q\in [0,2]$ and change the term  $D_2z_2$ to $D_2z_2^{q/2} z_1^{2-q}$. For $q=0$ we obtain the mean-field upper bound and for $q=2$ we get back the original lower bound. In the quadratic case the modified differential equations for the lower bound of the SIS epidemic \eqref{zsystem1SISa}--\eqref{zsystem2SISa}
without airborne infection, $\beta=0$, are
\begin{align}\label{zsystem1SISq}
z_1'&=D_0 + D_1 z_1 + D_2z_2^{q/2} z_1^{2-q} \\
\label{zsystem2SISq}
z_2'&=2D_0 z_2^{1/2} + 2D_1 z_2 +2 D_2 z_2^{3/2} +\frac{c}{N}.
\end{align}
In Figure \ref{fig:SISq} the solution of this system is shown for different values of $q$ together with the solution of the mean-field equation and the expected value. The curve for $q=0.5$ is close to the mean-field upper bound (i.e., $q=0$) and at least appears to be a valid lower bound for $y_1$.

An alternate approach for an improved lower bound modifies the differential equation for $z_2$ \eqref{zsystem2} by replacing the $z_2^{3/2}$ term by $z_2^2/z_1$ resulting in the system
\begin{align}\label{zsystem1SISCS}
z_1'&=D_0 + D_1 z_1 + D_2z_2 \\
\label{zsystem2SISCS}
z_2'&=2D_0 z_2^{1/2} + 2D_1 z_2 +2 D_2 z_2^2/z_1 +\frac{c}{N}.
\end{align}
By Jensen's inequality, $y_2^{3/2}\leq y_3$.
Thus assuming $z_2\approx y_2$, the $z_2^{3/2}$ term can be interpreted as a lower bound for a $y_3$ term.  Now we motivate the $z_2^2/z_1$ term. Using the Cauchy-Schwarz inequality, $y_2^2\leq y_3y_1$, leads to a lower bound $y_2^2/y_1\leq y_3$ that appears to be tighter than $y_2^{3/2}$.  Assuming $z_1\approx y_1$ and $z_2\approx y_2$ then motivates using $z_2^2/z_1$ instead of $z_2^{3/2}$.  From Figure \ref{fig:SISq} it appears to be a valid lower bound for $y_1$ and closer to the mean-field upper bound than the original lower bound.
Of course these potential lower bounds \eqref{zsystem1SISq}--\eqref{zsystem2SISq} and \eqref{zsystem1SISCS}--\eqref{zsystem2SISCS} are only justified by a numerical experiment (Figure \ref{fig:SISq}) and some intuition.  It is an open question whether these are provable lower bounds or not.

\end{document}